\documentclass[3p,times]{elsarticle}

\usepackage{amssymb}

\usepackage{amsmath,amsthm}
\usepackage{amssymb,latexsym}
\usepackage{enumerate}

\numberwithin{equation}{section}

\newtheorem{theorem}{Theorem}[section]

\newtheorem{proposition}[theorem]{Proposition}
\newtheorem{lemma}[theorem]{Lemma}
\newtheorem{definition}[theorem]{Definition}

\newtheorem{The main theorem}[theorem]{The main theorem}

\theoremstyle{definition}

\allowdisplaybreaks

\begin{document}

\begin{frontmatter}




\title{On the approximation of weakly   plurifinely plurisubharmonic functions}


\author[label1]{Nguyen Xuan Hong}
\address[label1]{Department of Mathematics, Hanoi National University of Education, 136 Xuan Thuy Street,
Cau Giay District, Hanoi, Vietnam}
\ead{xuanhongdhsp@yahoo.com}

\author[label3]{Hoang Van Can}
\address[label3]{Department of Basis Sciences, 
University of Transport Technology, 54 Trieu Khuc, Thanh Xuan District, Hanoi, Vietnam}  
\ead{ vancan.hoangk4@gmail.com}

\begin{abstract}
In this note, we study   the approximation of singular  plurifinely plurisubharmonic function $u$ defined on a plurifinely domain $\Omega$. Under some conditions, we prove that $u$ can be approximated  by an increasing  sequence of   plurisubharmonic functions defined  on Euclidean neighborhoods of  $\Omega$.  
\end{abstract}

\begin{keyword}
complex variables \sep
plurifinely pluripotential theory   \sep plurifinely plurisubharmonic functions


 \MSC[2010] 32U05 \sep 32U15
\end{keyword}

\end{frontmatter}



\section{Notation and main result}


Let $D$ be an open set in $\mathbb C^n$ and let   $ PSH^-(D)$ be the family of negative plurisubharmonic functions in $D$. 
The plurifine topology $\mathcal F$ on a  Euclidean open set $D$  is the smallest topology that makes all plurisubharmonic functions on $D$ continuous. 
Notions pertaining to the plurifine topology are indicated with the prefix $\mathcal F$ to distinguish them from notions pertaining to the Euclidean topology on $\mathbb C^n$.
For a set $A\subset \mathbb  C^n$ we write $\overline{A}$ for the closure of $A$ in the one point compactification
of $\mathbb C^n$, $\overline{A}^{\mathcal F}$ for the $\mathcal F$-closure of $A$ and $\partial _{\mathcal F}A$ for the $\mathcal F$-boundary of $A$.

Let $\Omega$ be  a bounded  $\mathcal F$-domain in   $\mathbb C^n$. A function $u : \Omega \to  [-\infty, +\infty) $  is said to be  $\mathcal F$-plurisubharmonic if $u$ is $\mathcal F$-upper semicontinuous and for every complex line $l$ in $\mathbb C^n$, the restriction of $u$ to any $\mathcal F$-component of the finely open subset $l \cap \Omega$  of $l$ is either finely subharmonic or $\equiv -\infty$.  
El Kadiri,  Fuglede and Wiegerinck \cite{KFW11} proved   the most important properties of the  $\mathcal F$-plurisubharmonic functions.  
El Kadiri and Wiegerinck \cite{KW14}   defined   the complex Monge-Amp\`{e}re operator for finite $\mathcal F$-plurisubharmonic functions on an $\mathcal F$-domain $\Omega$. Recently,    Hong and coauthors have been successfully pushing the theory of $\mathcal F$-plurisubharmonic functions (see   \cite{Hong17},  \cite{HHV17},  \cite {HV16},  \cite{TVH}).    
The aim of this note is to study  the 
conditions on $u$ and  $\Omega$ such that $u$ can be approximated  by an increasing  sequence of   plurisubharmonic functions defined  on Euclidean neighborhoods of  $\Omega$.   

When $\Omega$  is bounded Euclidean  domain with  $\mathcal C^1$-boundary.  Forn\ae ss and  Wiegerinck  \cite{FW89} proved that if $u$ is continuous on $\overline{\Omega}$ then $u$ can be approximated uniformly on $\overline{\Omega}$  by a sequence of smooth plurisubharmonic functions defined  on Euclidean neighborhoods of  $\Omega$.  

When $\Omega$  is bounded hyperconvex  domain.  According to the results by \cite{Be06}, \cite{Be11}, \cite{CH08}, \cite{Hed10} and other authors,  the approximation is possible if the  domain $\Omega$ has  the $\mathcal F$-approximation property and  $u$  belongs to one of  the Cegrell's classes in $\Omega$.  

When $\Omega$ is bounded $\mathcal F$-domain.
In   research \cite{TVH}, the authors gave the kind of $\Omega$ and $u$ that are in line with the $\mathcal F$-set   up to make the approximation possible.  

The purpose of this note  is to  extend  the result of \cite{TVH}.   
In analogy with the  set up of the hyperconvex domain to make the approximation possible, we introduce the following.

\begin{definition}{\rm
Let  $\Omega$ be a bounded  $\mathcal F$-hyperconvex domain, 
 i.e.,  it is  a bounded, connected, and $\mathcal F$-open set such that  there exist a negative bounded plurisubharmonic function $\gamma_\Omega$ defined in a bounded hyperconvex domain $\Omega'$ such that $\Omega =\Omega' \cap \{\gamma_\Omega >-1\}$ and $-\gamma_\Omega $ is $\mathcal F$-plurisubharmonic  in $\Omega$.
We say that $\Omega$ has  the $\mathcal F$-approximation property if  there exist an increasing  sequence of negative  plurisubharmonic functions  $ \rho_j $  defined on  bounded  hyperconvex domains $\Omega_j$ such that  $\Omega\subset\Omega_{j+1}\subset\Omega_j$ and   $\rho_j\nearrow \rho \in \mathcal E_0(\Omega)$ a.e. on $\Omega$ as $j\nearrow+\infty$. Here    
\begin{align*}
\mathcal E_0(\Omega):=
\{ u \in \mathcal F\text{-} & PSH^-(\Omega) \cap L^\infty(\Omega):    \int_\Omega 
  (dd^c u)^n <+\infty 
  \\&  \text{ and  }  \forall \varepsilon>0, \ \exists  \delta>0,\  \overline{\Omega\cap  \{u<- \varepsilon\} } \subset  \Omega' \cap\{\gamma_\Omega >-1+\delta\}\} .
\end{align*}
}\end{definition}

Example 3.3 in \cite{TVH} showed that there exists 
a bounded $\mathcal F$-hyperconvex domain $\Omega$ that has   the $\mathcal F$-approximation property, moreover, it
has no Euclidean interior point exists.
For the precise definition and properties of the class $\mathcal F(\Omega)$  we refer the reader to the next section. 
Our main result is  the following theorem. 

\begin{theorem} \label{the1}
Let  $\Omega$ be a bounded $\mathcal F$-hyperconvex domain and let $u\in\mathcal F  (\Omega)$. Assume that $\Omega$ has the $\mathcal F$-approximation property. Then,  there exists an increasing  sequence of    plurisubharmonic functions $u_j$ defined  on Euclidean neighborhoods of  $\Omega$ such that  $u_j\nearrow   u$ a.e. on $\Omega$ as $j \nearrow + \infty$.
\end{theorem}

The note is organized as follows. 
In Section 2,  we introduce and investigate the class $\mathcal F (\Omega)$. 
Section 3 is devoted to prove  Theorem \ref{the1}.

\section{The class $\mathcal F(\Omega)$}
Some elements of pluripotential  theory (plurifine potential  theory) that will be used  throughout  the paper can be  found  in \cite{ACCH}-\cite{W12}.  We denote by   $\mathcal F\text{-}  PSH^-(\Omega)$   the set of negative   $\mathcal F\text{-}$plurisubharmonic functions defined in $\mathcal F\text{-}$open set $\Omega$.
First, we recall the definition of the complex  Monge-Amp\`ere  measure for  finite  $\mathcal F\text{-}$plurisubharmonic functions.

\begin{definition}{\rm 
Let $\Omega$ be an $\mathcal F$-open set in $\mathbb C^n$ and let  $QB(\Omega)$ be the trace of $QB(\mathbb C^n)$ on $\Omega$, where $QB(\mathbb C^n)$  denotes   the   $\sigma$-algebra  on $\mathbb C^n$ generated by the Borel sets and the pluripolar subsets of  $\mathbb C^n$.
Assume that 
$u_1, \ldots , u_n \in \mathcal F\text{-} PSH(\Omega)$ are  finite. 
Using the quasi-Lindel\"of property of the plurifine topology and Theorem 2.17 in \cite{KW14}, there exist a pluripolar set $E\subset \Omega$, a sequence of $\mathcal F$-open subsets  $\{O_k\}$   and   plurisubharmonic functions  $f_{j,k}, g_{j,k}$ defined in Euclidean neighborhoods of $\overline O_k$  such that    $\Omega=E \cup \bigcup_{k=1}^\infty O_k $ and   $u_j=f_{j,k} -g_{j,k}$ on $O_k$. We define $O_0:= \emptyset$ and   
\begin{equation}\label{eqqqqqqqq}
\int_A dd^c u_1 \wedge \ldots \wedge dd^c u_n  
:= \sum_{k=1}^\infty \int_{A\cap ( O_k \backslash \bigcup_{j=0}^{k-1} O_j ) } dd^c ( f_{1,k} -g_{1,k}) \wedge \ldots  \wedge dd^c ( f_{n,k} -g_{n,k}), \  
A\in QB(\Omega). 
\end{equation} 
 Theorem 3.6 in \cite{KW14} implies that  the measure  defined by \eqref{eqqqqqqqq} is  independent on $E$,   $\{O_k\}$, $\{f_{j,k}\}$ and $\{g_{j,k}\}$.
This measure  is called the complex  Monge-Amp\`ere  measure. 
}\end{definition}

Note that from Theorem 2.17 in \cite{KW14} and Lemma 4.1 in \cite{KW14} we infer at $dd^c u_1 \wedge \ldots \wedge dd^c u_n  $ is a non-negative measure  on $QB(\Omega)$.  
We now give   the following definition which is an extension of the class $\mathcal F(\Omega)$ introduced and investigated by Cegrell \cite{Ce2} when  $\Omega$ is a bounded hyperconvex domain in $\mathbb C^n$.

\begin{definition}  
{\rm 
Let $\Omega$ be a bounded $\mathcal F$-hyperconvex domain in $\mathbb C^n$. We denote by   $\mathcal F  (\Omega)$  the family of negative $\mathcal F$-plurisubharmonic functions $u$ defined on $\Omega$ such that there exist a decreasing sequence $\{\varphi _j\} \subset \mathcal E_0(\Omega)$ that converges pointwise to $u$ on $\Omega$ and 
$$\sup_{j\geq 1} \int_\Omega  (dd^c \varphi_j )^n<+\infty.$$

\noindent 
Furthermore, if $p>0$ satisfies $$\sup_{j\geq 1} \int_\Omega (1+(-\varphi_j)^p) (dd^c \varphi_j )^n<+\infty $$ 
then we say that $u \in \mathcal F_p(\Omega)$.
} \end{definition}

Note that 
$\mathcal F(\Omega) \cap L^\infty(\Omega) \subset \mathcal F_p(\Omega) \subset \mathcal F(\Omega)$ for all $p>0$.

\begin{proposition} \label{pro2-1-1}
Let $\Omega \Subset \mathbb C^n$ be a bounded $\mathcal F$-hyperconvex domain in $\mathbb C^n$ and let  $u\in \mathcal F  (\Omega) \cap L^\infty(\Omega)$. Then, the statements are holds. 

(i) If     $\{\varphi_j \} \subset \mathcal E_0(\Omega)$ such that $\varphi_j  \searrow u$ on $\Omega$ and 
$\sup_{j\geq 1} \int_\Omega (dd^c \varphi_j )^n<+\infty$  then  
$$\int_\Omega (- \rho )  (dd^c u)^n 
= \sup_{j\geq1}\int_\Omega (- \rho ) (dd^c u_j)^n, \ \ \forall \rho \in \mathcal F\text{-}PSH^-(\Omega) \cap L^\infty(\Omega).$$ 

(ii) If  $v\in \mathcal F\text{-}PSH(\Omega)$ with $u\leq v<0$ then $v\in \mathcal F   (\Omega)$ and 
$\int_{\Omega }   (dd^c v)^n \leq  \int_{\Omega}   (dd^c u)^n.$
\end{proposition}

\begin{proof}
The statement  follows from Proposition 4.2 in \cite{TVH} and Proposition 4.3 in \cite{TVH}.
\end{proof}

\begin{proposition}  \label{pro99991}
Let $\Omega$ be a bounded $\mathcal F$-hyperconvex domain in $\mathbb C^n$ and let   $u \in \mathcal F  (\Omega)$.  If   $\{u_j \} \subset \mathcal F(\Omega) \cap L^\infty (\Omega) $ such that $u_j \searrow u$ in $\Omega$ as $j\nearrow+\infty$ then
$$\sup_{j\geq 1} \int_\Omega  (dd^c u_j)^n<+\infty.$$  and  
$$
\int_\Omega (dd^c \max(u, \rho))^n 
= \sup_{j\geq 1} \int_\Omega  (dd^c u_j)^n
$$
for every  $\rho \in \mathcal F \text{-} PSH^-(\Omega) \cap L^\infty(\Omega)$ with  $\sup_\Omega \rho <0$.
In particular, 
$$
\int_\Omega (dd^c \max(u,-1))^n <+\infty.
$$
\end{proposition}

\begin{proof}  
Let $\{\varphi_k  \} \subset \mathcal E_0(\Omega)$ such that $\varphi _k  \searrow u$ in $\Omega$ as $k \nearrow+\infty$ and 
$$\sup_{k \geq 1} \int_\Omega  (dd^c \varphi_k )^n<+\infty.$$ 
Since $\max(u_j , \rho_k) \searrow u_j$ in $\Omega$ as $k\nearrow+\infty$, by Proposition 3.4 in \cite{TVH} and Proposition 4.2 in \cite{TVH} we infer at 
\begin{align*}
\int_\Omega  (dd^c u_j)^n 
= \sup_{k\geq 1} \int _\Omega (dd^c \max(u_j, \varphi_k ))^n 
\leq \sup_{k \geq 1} \int_\Omega  (dd^c \varphi_k )^n. 
\end{align*}
This implies that 
$$\sup_{j\geq 1} \int_\Omega  (dd^c u_j)^n  
\leq \sup_{k \geq 1} \int_\Omega  (dd^c \varphi_k )^n <+\infty.$$

\noindent 
Now, assume that $\rho \in \mathcal F \text{-} PSH^-(\Omega) \cap L^\infty(\Omega)$,  $\sup_\Omega \rho <0$. Thanks to  Proposition 3.4 in \cite{TVH} and   Proposition \ref{pro2-1-1} we have 
\begin{align*}
\int _\Omega (dd^c \max(u, \rho ))^n
& = \sup_{k\geq 1} \int _\Omega (dd^c \max(\varphi _k ,  \rho ))^n
\\& = \sup_{k\geq 1} \int _\Omega (dd^c  \varphi _k )^n
\\& = \sup_{k\geq 1} \sup_{j\geq 1} \int _\Omega (dd^c  \max( u_j,\varphi _k ))^n
=\sup_{j\geq 1} \int_\Omega  (dd^c u_j)^n.
\end{align*}
The proof is complete.
\end{proof}

\begin{proposition} \label{pro4hh}
Let $\Omega$ be a bounded $\mathcal F$-hyperconvex domain in $\mathbb C^n$. Assume that  $u\in \mathcal F  (\Omega) \cap L^\infty(\Omega)$ and $v\in \mathcal F\text{-}PSH^-(\Omega)$ such that   $ (dd^c u)^n \leq  (dd^c v)^n$ in $\Omega \cap   \{v>-\infty\}$.  
Then, $u \geq v$ in $\Omega$. 
\end{proposition}

\begin{proof}
Without loss of generality  we can assume that $-1\leq u \leq 0$ on $\Omega$. Let $j \in \mathbb N^*$ and define 
$$
v_j:= (1+\frac{1}{j}) (v-\frac{1}{j}) \text{ in } \Omega.
$$
Choose $p>0$ such that 
$ j^p <  1+\frac{1}{j} $.
It is easy to see that 
\begin{align*}
(1+ (-u)^p) (dd^c u)^n 
\leq 2 (dd^c u)^n  
\leq 2 (dd^c v)^n 
\leq (1+ (-v_j)^p) (dd^c v_j)^n 
\text{ on } \Omega \cap   \{v_j>-\infty\}.
\end{align*}
Proposition 4.4 in \cite{TVH} implies that 
$u \geq v_j$ in $\Omega$. 
Letting $j\to+\infty$ we conclude  that $u \geq v$ in $\Omega$. The proof is complete. 
\end{proof}

\section{Proof of Theorem \ref{the1}}

We need the following.

\begin{lemma}  \label{lem1}
Let $\Omega$ be a bounded $\mathcal F$-hyperconvex domain in $\mathbb C^n$ and let  $u ,v \in \mathcal F  (\Omega)$ be such that 

(i) $u\geq v$ in $\Omega$;

(ii)  $(dd^c u)^n \leq   (dd^c v)^n$ on $\Omega \cap \{v>-\infty\}$;

(iii)   $\int_\Omega (dd^c \max(u,-1))^n \geq \int_\Omega (dd^c \max(v,-1))^n $.

\noindent 
Then, $u=v$ in $\Omega$.
\end{lemma}

\begin{proof}
Let $R>0$ be such that $\Omega \Subset \mathbb B(0,R)$ and define  $\rho(z):=|z|^2 - R^2$, $z\in \mathbb C^n$. 
Let $\varepsilon ,\delta \in (0, 1)$. We set  
$$
u_{\varepsilon, \delta} : = \max(  u , - 2 \delta R^2 \varepsilon ^{-1})
\text{ and } 
v_{\varepsilon, \delta} : = \max( (1-\varepsilon) u _{\varepsilon, \delta} + \delta \rho , v).
$$ 
Since $\sup_\Omega \rho <0$ and  $u \geq  v$ in $\Omega$, by Proposition   \ref{pro2-1-1}  and Proposition \ref{pro99991} we conclude by  (iii) that  
\begin{equation} \label{eq1.24.10.11}
\begin{split}
\int_\Omega (dd^c v_{\varepsilon, \delta})^n 
& =\int_\Omega (dd^c \max(v,-1))^n 
\\& = \int_\Omega (dd^c \max(u,-1))^n 
=\int_\Omega (dd^c u_{\varepsilon, \delta})^n .
\end{split}
\end{equation}

\noindent 
Since $u_{\varepsilon, \delta} =u$ on $\{v>-2 \delta R^2 \varepsilon^{-1}\}$,  by   Theorem 4.8 in \cite{KW14} and using (ii), we get 
$$
(dd^c v )^n \geq (dd^c u )^n = (dd^c u_{\varepsilon, \delta} )^n 
\text{ on } \{v> - 2 \delta R^2 \varepsilon^{-1} \}.
$$
Hence, Proposition 2.6 in \cite{TVH} implies that 
\begin{equation} \label{eq1}
(dd^c v_{\varepsilon, \delta})^n 
\geq (1-\varepsilon)^n (dd^c u_{\varepsilon, \delta} )^n  
\text{ on } \{v> - 2 \delta R^2 \varepsilon^{-1} \}. 
\end{equation} 
Because 
$v_{\varepsilon, \delta}  = (1-\varepsilon) u_{\varepsilon, \delta} + \delta \rho$ in $\{v<  - \delta R^2 \varepsilon^{-1} \} \cup \{ v< u - \delta R^2\}$,   by Theorem 4.8 in \cite{KW14} we infer that 
$$
(dd^c v_{\varepsilon, \delta})^n \geq  (1-\varepsilon)^n (dd^c u_{\varepsilon, \delta} )^n  + \delta^n  (dd^c \rho)^n
\text{ on } \{v<  - \delta R^2 \varepsilon^{-1} \} \cup \{ v< u - \delta R^2\}.
$$ 
Combining this with \eqref{eq1}   we arrive at 
\begin{align*}  
(dd^c v_{\varepsilon, \delta})^n 
\geq (1-\varepsilon)^n  (dd^c   u_{\varepsilon, \delta}   )^n   
+ \delta^n 1_{\{ v< u - \delta R^2\}} (dd^c \rho)^n
\text{ on } \Omega.
\end{align*}
It follows that 
\begin{align*}
\int_\Omega (dd^c v_{\varepsilon,\delta})^n 
 \geq (1-\varepsilon)^n \int_\Omega (dd^c  u _{\varepsilon,\delta} )^n  
+ \delta^n \int_ {\{ v< u- \delta R^2\}} (dd^c \rho)^n .
\end{align*}
Letting $\varepsilon \to 0$ we conclude by \eqref{eq1.24.10.11} that 
$$
 \int_ {\{ v< u - \delta R^2\}} (dd^c \rho)^n 
 =0, \ \forall \delta>0.
 $$
Therefore, by Proposition 2.3 in \cite{TVH} we infer that $v\geq u$ in $\Omega$, and hence, $u=v$ in $\Omega$. The proof is complete.  
\end{proof}

We now able to give the proof of Theorem \ref{the1}.

\begin{proof}
[Proof of Theorem \ref{the1}]   
Since   $\Omega$ has  the $\mathcal F$-approximation property, so   there exist an increasing  sequence of negative  plurisubharmonic functions  $ \rho_j $  defined on  bounded  hyperconvex domains $\Omega_j$ such that  $\Omega\subset\Omega_{j+1}\subset\Omega_j$ and   $\rho_j\nearrow \rho \in \mathcal E_0(\Omega)$ a.e. on $\Omega$.  
Let $k\in \mathbb N$ be such that $k\geq 1$. Proposition \ref{pro99991} implies that 
$$
\int_\Omega (dd^c \max(u,-k))^n
=\int_\Omega (dd^c \max(u,-1))^n<+\infty.
$$
Since the   measure $1_{\Omega}  (dd^c \max(u, - k) )^n$ vanishes on all pluripolar subsets of $\Omega_j$, by Lemma 5.14  in \cite{Ce2} there exists $u_{j,k} \in \mathcal F (\Omega_j)$ such that 
$$  
(dd^c u_{j,k})^n=1_{ \Omega  }  (dd^c \max(u,-k) )^n \text{ in } \Omega_j. 
$$
Theorem 3.7 in \cite{KH} states that the   function  $u_j:=(\limsup_{k\to+\infty} u_{j,k})^*$ belongs to $\mathcal F(\Omega_j)$, where $^*$ denotes the upper semi-continuous regularization.
By Theorem 5.5 in \cite{Ce2} and Proposition \ref{pro4hh} we infer that 
$u_{j,k} \leq u_{j+1,k} \leq \max(u,-k) \text{ on } \Omega $, and hence,  
$$u_j \leq u_{j+1} \leq u \text{ on }\Omega.$$  
We now claim that 
\begin{equation} \label{9.27.11.11}
(dd^c u_j)^n \geq (dd^c u)^n \text{ on } \Omega\cap \{u>- \infty\} 
\end{equation}
and  
\begin{equation} \label{9.28.11.11}
\int_{\Omega_j} (dd^c u_j)^n 
\leq  \int_\Omega (dd^c \max(u,-1))^n.
\end{equation}
Indeed, fix $a>0$ and let $k\in \mathbb N^*$ be such that $k\geq a$.    
Since 
$$  
(dd^c u_{j,k+s})^n \geq 1_{ \Omega  \cap \{u>-a\} }  (dd^c u)^n \text{ in } \Omega_j, \ \forall s\geq 0,
$$
Proposition 4.3 in \cite{KH} implies that 
$$(dd^c  \max (u_{j,k}, \ldots, u_{j,k+s}))^n  \geq 1_{ \Omega  \cap \{u>-a\} }  (dd^c u)^n \text{ in } \Omega_j, \ \forall s\geq 0.$$
Main Theorem in \cite{Ce3} states that 
$$
(dd^c (\sup_{l\geq 0} u_{j,k+l})^*)^n  \geq 1_{ \Omega  \cap \{u>-a\} }  (dd^c u)^n \text{ in } \Omega_j 
$$ 
because $\max (u_{j,k}, \ldots, u_{j,k+s}) \nearrow (\sup_{l\geq 0} u_{j,k+l})^*$ a.e. in $\Omega_j$ as $s\nearrow+\infty$. Moreover, since $(\sup_{l\geq 0} u_{j,k+l})^* \searrow u_j$ a.e. in $\Omega_j$ as $k\nearrow+\infty$, again by Main Theorem in \cite{Ce3} we infer that 
$$
(dd^c u_j )^n  \geq 1_{ \Omega  \cap \{u>-a\} }  (dd^c u)^n \text{ in } \Omega_j .
$$
Letting $a\to+\infty$, we get
$$
(dd^c u_j)^n \geq (dd^c u)^n \text{ on } \Omega\cap \{u>- \infty\} .
$$
Now, by  Lemma 3.3 in \cite{ACCH} and Corollary 3.4 in \cite{ACCH} we have 
\begin{align*}
\int_{\Omega_j} (dd^c u_j)^n 
& =\lim_{k\to+\infty} \int_{\Omega_j} (dd^c (\sup_{l\geq 0} u_{j,k+l})^*)^n 
\\& \leq \sup_{k \geq 1}  \int_{\Omega_j} (dd^c u_{j,k})^n 
=\int_\Omega (dd^c \max(u,-1))^n.
\end{align*}
This proves the claim. Let $v$  be the least $\mathcal F$-upper semicontinuous majorant of $\sup_{j\geq 1} u_j$ in $\Omega$. 
Then, $v \in \mathcal F\text{-} PSH^{-}(\Omega)$ and $v\leq u$ on $\Omega $.
By Theorem 4.5 in \cite{KS14} and using  \eqref{9.27.11.11} we infer that 
\begin{equation} \label{9.36.11.11}
(dd^c v)^n \geq (dd^c u)^n \text{ on } \Omega\cap \{v>- \infty\}.
\end{equation}

\noindent 
We claim that $v\in \mathcal F  (\Omega)$. Indeed, 
put $v_k:=\max(v,k \rho)$, where $k\in \mathbb N^*$.  
Proposition 3.4 in \cite{TVH} implies that  $v_k\in \mathcal E_0(\Omega)$. Since $\max(u_j, k\rho_j) \nearrow v_k$ a.e. in $\Omega$ as $j\nearrow+\infty$, by    Proposition 2.7 in \cite{TVH} and     Lemma 3.3 in \cite{ACCH} we obtain  by \eqref{9.28.11.11} that 
\begin{align*}
\int_\Omega  (dd^c v_k)^n 
& \leq \liminf_{j\to+\infty} \int_\Omega  (dd^c \max(u_j, k\rho_j))^n  
\\ & \leq \liminf_{j\to+\infty} \int_{\Omega_j} (dd^c u_j)^n 
\leq \int_\Omega (dd^c \max(u,-1))^n.
\end{align*}
Since $v_k \searrow v$, by Proposition \ref{pro99991} we obtain $v\in \mathcal F  (\Omega)$. This proves the claim. 
Now, again by Proposition 2.7 in \cite{TVH} and Proposition 3.4 in \cite{TVH} we have 
\begin{align*}
\int_\Omega (dd^c \max(v,-1))^n
& \leq \liminf_{k\to+\infty} \int_\Omega  (dd^c \max(v_k,-1))^n  
\\ & \leq \liminf_{k\to+\infty} \int_{\Omega}  (dd^c v_k)^n  
\leq \int_\Omega (dd^c \max(u,-1))^n.
\end{align*}
Combining this with \eqref{9.36.11.11} and using   Lemma \ref{lem1}  we conclude  that 
$v=u$ in $\Omega$. 
Thus,  $u_j \nearrow u$ a.e. in $\Omega$ as $j  \nearrow+\infty$. The proof is complete.   
\end{proof}








\end{document}